\documentclass[12pt]{article}
\usepackage{amsmath}
\usepackage{latexsym}
\usepackage{algorithm}
\usepackage{algpseudocode}
\usepackage{tikz}

\usepackage{url}
\usepackage{amssymb}
\usepackage{exscale}
\usepackage{graphicx}
\usepackage{makeidx}

\usepackage{color}
\usepackage{pstricks}

\usepackage{float}
\floatstyle{plaintop}
\restylefloat{table}

\makeatletter
\def\BState{\State\hskip-\ALG@thistlm}
\makeatother

\newtheorem{thm}{Theorem}
\newtheorem{cor}[thm]{Corollary}
\newtheorem{con}[thm]{Conjecture}

\newtheorem{lem}[thm]{Lemma}

\newenvironment{proof}[1][Proof]
{\par\noindent{\bf #1.} }{\hspace*{\fill}\nolinebreak{$\Box$}\bigskip\par}

\title{\bf Restricted size Ramsey number for $P_3$ versus cycle}
\author{
Joanna Cyman\\
\small Gda\'nsk University of Technology \\[-0.8ex]
\small Department of Technical Physics and Applied Mathematics \\[-0.8ex]
\small Narutowicza 11/12, 80-952 Gda\'nsk, Poland\\[-0.8ex]
\small \texttt{joanna.cyman@pg.edu.pl}\\
\\[-1.5ex]
and\\[-1.5ex]
\\
Tomasz Dzido\\
\small Institute of Informatics, University of Gda\'{n}sk \\[-0.8ex]
\small Wita Stwosza 57, 80-952 Gda\'{n}sk, Poland\\[-0.8ex]
\small \texttt{tdz@inf.ug.edu.pl}\\
}

\begin{document}
\maketitle
\begin{abstract}

\noindent
Let $F$, $G$ and $H$ be simple graphs. We say $F \rightarrow (G, H)$ if for every \linebreak  $2$-coloring of the edges of $F$ there exists a red copy of $G$ or a blue copy of $H$ in $F$. The Ramsey number $r(G, H)$ is defined as $r(G, H) = min\{|V (F)|: F \rightarrow (G, H)\}$, while the restricted size Ramsey number $r^{*}(G, H)$ is defined as $r^{*}(G, H) = min\{|E (F)|: F \rightarrow (G, H) , |V (F) | = r(G, H)\}$. In this paper we determine previously unknown restricted size Ramsey numbers $r^{*}(P_3, C_n)$ for $7 \leq n \leq 12$. We also give new upper bound $r^{*}(P_3, C_n) \leq 2n-2$ for $n \geq 10$ and $n$ is even.
\end{abstract}

\section{Introduction}

\noindent Paul Erd\H os had a tremendous impact on many areas of mathematics, one of these areas is Ramsey theory.  His contributions started with the classical Ramsey numbers $r(G,H)$.  In 1978 Erd\H os  \emph{et al.} in  \cite{efrs} defined  the \emph{size Ramsey number} $\hat{r}(G,H)$ as the smallest size of a graph $F$ such that, under any 2-coloring of its edges, the graph $F$ contains a red copy of $G$ or a blue copy of $H$. In \cite{fs} one can find a survey of results along with the influence of Paul Erd\H os on the development of size Ramsey theory.

The \emph{restricted size Ramsey number} $r^{*}(G,H)$ is a problem connecting Ramsey number and size Ramsey number.  For the restricted size Ramsey number, if $r$ is the Ramsey number of $G$ and $H$ then $F$ must be a spanning subgraph of $K_r$ with the smallest size such that for any 2-coloring of edges of $F$ we have a red copy of $G$ or a blue copy of $H$ in $F$.  Therefore, the size of $K_r$ is the upper bound for the restricted size Ramsey number of $G$ and $H$ and the restricted size Ramsey number must be greater or equal to the size Ramsey number for a given pair of graphs. In addition, we have $\tilde{r}(G,H) \leq \hat{r}(G,H)$, where $\tilde{r}(G,H)$ is the on-line Ramsey number.
If both $G$ and $H$ are complete graphs then $F=K_r$ (see \cite{efrs}). The case of complete graph is one of a few cases for which that upper bound is reached. In general, the more sparse both graphs $G$ and $H$ are, the problem of finding the restricted size Ramsey number for those pair of graphs is harder. Only two results for the exact value of restricted size Ramsey number involving a class of graph known so far, that are, for $K_{1,k}$ versus $K_n$ \cite{fsh} and $G$ versus $K_{1,k}$, where $G$ is $K_3$, $K_4-e$, or $C_5$ \cite{ef}. For other few classes of graphs, the problem is solved partially.

Some results for size Ramsey number was presented by Faudree and Schelp in 2002 \cite{fs}. It had shown that $r^{*}(P_3,C_3)=8$, $r^{*}(P_3,C_4)=6$, $r^{*}(P_3,C_5)=9$. In 2015 Silaban \emph{et al.} proved the last known exact value, namely $r^{*}(P_3,C_6)=9$ \cite{dbu1}. In addition, they give  lower and upper bound for $r^{*}(P_3,C_n)$, where $n \geq 8$ is even (see below Theorem \ref{tw3}). In this paper, we determine previously unknown restricted size Ramsey numbers, namely $r^{*}(P_3, C_n)$ for $7 \leq n \leq 12$, and we improve the upper bound for $r^{*}(P_3,C_n)$, that is we prove that $r^{*}(P_3, C_n) \leq 2n-2$ for $n \geq 10$ and $n$ is even.

For notation and graph theory terminology we in general follow \cite{dbu1}.

\section{Known results}

In this section, we list a few known definitions and theorems that we will need in proving our results.

The \emph{Tur{\'a}n number} $ex(n,G)$ is the maximum number of edges in any $n$-vertex graph which does not contain a subgraph isomorphic to $G$.  A graph on $n$ vertices is said to be \emph{extremal with respect to $G$} if it does not contain a subgraph isomorphic to $G$ and has exactly $ex(n,G)$ edges.

In 1989 Clapham \emph{et al.} \cite{cfs} determined all values of $ex(n,C_4)$ for $n \leq 21$. They also characterized all the corresponding extremal graphs. In Theorem 1 we quote value for $ex(7, C_4)$ and we show (Figure \ref{rys1}) five corresponding extremal graphs.
We will use this in the proof of Theorem~5.

\begin{thm}[\cite{cfs}]
$$ex(7, C_4)=9$$ and there are 5 extremal graphs for this number illustrated in Figure \ref{rys1}.
\label{tw2}
\end{thm}
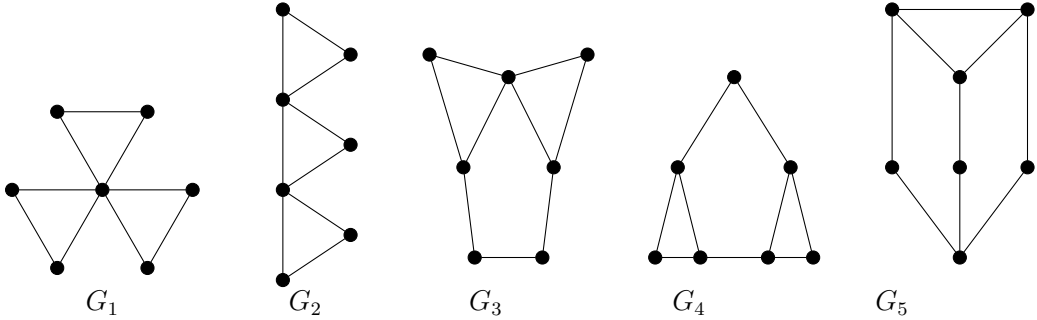
\begin{figure}

\begin{tikzpicture}[scale=.15, transform shape]
\tikzstyle{every node}=[draw, circle, minimum size=33 pt, fill=black];
\node (v0) at (0:0) {};
\node (v1) at (   0:8) {} ;
\node (v2) at (  60:8) {};
\node (v3) at (120:8) {};
\node (v4) at (180:8) {};
\node (v5) at (240:8) {};
\node (v6) at (300:8) {};
\draw (v0) -- (v1)
(v0) -- (v2)
(v0) -- (v3)
(v0) -- (v4)
(v0) -- (v5)
(v0) -- (v6)
(v1) -- (v6)
(v2) -- (v3)
(v4) -- (v5);

\node (v7) at (0:16) {};
\node (v8) at (22,4) {} ;
\node (v9) at (16,8) {};
\node (v10) at (22,12) {};
\node (v11) at (16,16) {};
\node (v12) at (22,-4) {};
\node (v13) at (16,-8) {};
\draw (v7) -- (v13)
(v7) -- (v8)
(v7) -- (v9)
(v9) -- (v10)
(v9) -- (v11)
(v7) -- (v12)
(v8) -- (v9)
(v10) -- (v11)
(v12) -- (v13);

\node (v14) at (32,2) {};
\node (v15) at (40,2) {} ;
\node (v16) at (29,12) {};
\node (v17) at (36,10) {};
\node (v18) at (43,12) {};
\node (v19) at (33,-6) {};
\node (v20) at (39,-6) {};
\draw (v14) -- (v16)
(v14) -- (v17)
(v15) -- (v17)
(v15) -- (v18)
(v16) -- (v17)
(v17) -- (v18)
(v14) -- (v19)
(v19) -- (v20)
(v15) -- (v20);

\node (v21) at (49,-6) {};
\node (v22) at (53,-6) {} ;
\node (v23) at (59,-6) {};
\node (v24) at (63,-6) {};
\node (v25) at (51,2) {};
\node (v26) at (61,2) {};
\node (v27) at (56,10) {};
\draw (v21) -- (v22)
(v22) -- (v23)
(v23) -- (v24)
(v21) -- (v25)
(v22) -- (v25)
(v23) -- (v26)
(v24) -- (v26)
(v25) -- (v27)
(v26) -- (v27);

\node (v28) at (76,-6) {};
\node (v29) at (70,2) {} ;
\node (v30) at (76,2) {};
\node (v31) at (82,2) {};
\node (v32) at (76,10) {};
\node (v33) at (70,16) {};
\node (v34) at (82,16) {};
\draw (v28) -- (v29)
(v28) -- (v30)
(v28) -- (v31)
(v30) -- (v32)
(v31) -- (v34)
(v32) -- (v33)
(v32) -- (v34)
(v33) -- (v34)
(v29) -- (v33);

\tikzstyle{every node}=[x=1cm, y=1cm];

\draw[color=black] (0,-10) node[scale=6.0] {$G_1$};
\draw[color=black] (18,-10) node[scale=6.0] {$G_2$};
\draw[color=black] (34,-10) node[scale=6.0] {$G_3$};
\draw[color=black] (52,-10) node[scale=6.0] {$G_4$};
\draw[color=black] (70,-10) node[scale=6.0] {$G_5$};


\end{tikzpicture}
\caption{All extremal graphs for $ex(7,C_4)$.}
\label{rys1}
\end{figure}


\medskip
In our work we will also use a well known the Ramsey number for paths and cycles that was calculated by Faudree \emph{et al.} in \cite{flps}.

\begin{thm}[\cite{flps}]
For all integers $n \geq 4$,
$$r(P_3,C_n)= n.$$
\label{tw1}
\end{thm}

In 2015, Silaban \emph{et al.} \cite{dbu1} proved the lower and the upper bound for the restricted size Ramsey number for $P_3$ and cycles. At the end of our article we improve the upper bound for this number.

\begin{thm}[\cite{dbu1}]
For $n \geq 8$, $n$ is even,
$$\frac{3}{2}n+2 \leq r^{*}(P_3, C_n) \leq 2n-1.$$
\label{tw3}
\end{thm}

\section{New results}

In order to find the value of $r^{*}(P_3,C_n)$, we must find a graph $F$ with the smallest possible size such that $F \rightarrow (P_3, C_n)$. According to Theorem \ref{tw1} the graph $F$ must have $n$ vertices.

\subsection{Determining the value of $r^{*}(P_3,C_7)$}

First, we give the following condition for graph $F$ satisfying $F \rightarrow (P_3, C_7)$.

\begin{lem}
Let $F$ be a graph with $|V(F)|=7$ and $C_4 \subseteq \overline{F}$, then $F \nrightarrow (P_3, C_7)$.
\label{lem1}
\end{lem}

\begin{proof}
Suppose there is $F$ with $|V(F)|=7$ such that $\overline{F}$ contains cycle $C_4$, say $v_1, v_2, v_3, v_4, v_1$. By coloring possible edges $v_1v_3$, $v_2v_4$ $\in E(F)$ in red and the remaining edges of $F$ in blue, we obtain a $2$-coloring of $F$ which contains neither a red $P_3$ nor a blue $C_7$.
\end{proof}

\begin{thm}
$r^{*}(P_3,C_7)=13$.
\label{twC7}
\end{thm}

\begin{proof}
First, we will prove that $r^{*}(P_3,C_7) \geq 13$.  From Lemma \ref{lem1} and Theorem \ref{tw2} we imply that $r^{*}(P_3,C_7) \geq 12$. Suppose that $r^{*}(P_3,C_7) = 12$. Let $F$ be a graph on $7$ vertices and $12$ edges. By Lemma \ref{lem1}, if $F \rightarrow (P_3, C_7)$, then $C_4\nsubseteq \overline{F}$ and therefore $\overline{F}$ is one of the five graphs $G_i$, $1 \leq i \leq 5$ from Figure~\ref{rys1}.
Furthermore, since $\Delta(G_i) \geq 4$ for $i \in \{1,2,3\}$, and by coloring $u_2u_7$, $u_3u_5$, $v_1v_6$, $v_3v_7$, $v_4v_5$ in red (see Figure \ref{rysG4G5}) we obtain, for all $\overline{G_i}$, a $2$-coloring of edges which contains neither a red $P_3$ nor a blue $C_7$. In fact, if $\Delta(G_i) \geq 4$, then there is a vertex of degree at most 2 in $\overline{G_i}$. To avoid a blue $C_7$ we color in red one edge coming out of this vertex (if any). Hence, $F \nrightarrow (P_3, C_7)$ and consequently  we have $r^{*}(P_3,C_7) \geq 13$.

\begin{figure}
\begin{center}
\begin{tikzpicture}[scale=.15, transform shape]
\tikzstyle{every node}=[draw, circle, minimum size=33 pt, fill=black];

\node (v21) at (49,-6) {};
\node (v22) at (53,-6) {} ;
\node (v23) at (59,-6) {};
\node (v24) at (63,-6) {};
\node (v25) at (51,2) {};
\node (v26) at (61,2) {};
\node (v27) at (56,10) {};
\draw (v21) -- (v22)
(v22) -- (v23)
(v23) -- (v24)
(v21) -- (v25)
(v22) -- (v25)
(v23) -- (v26)
(v24) -- (v26)
(v25) -- (v27)
(v26) -- (v27);

\node (v28) at (76,-6) {};
\node (v29) at (70,2) {} ;
\node (v30) at (76,2) {};
\node (v31) at (82,2) {};
\node (v32) at (76,10) {};
\node (v33) at (70,16) {};
\node (v34) at (82,16) {};
\draw (v28) -- (v29)
(v28) -- (v30)
(v28) -- (v31)
(v30) -- (v32)
(v31) -- (v34)
(v32) -- (v33)
(v32) -- (v34)
(v33) -- (v34)
(v29) -- (v33);

\tikzstyle{every node}=[x=1cm, y=1cm];

\draw[color=black] (56,-11) node[scale=6.0] {$G_4$};
\draw[color=black] (76,-11) node[scale=6.0] {$G_5$};

\draw[color=black] (49,-8) node[scale=5.0] {$u_1$};
\draw[color=black] (53,-8) node[scale=5.0] {$u_2$};
\draw[color=black] (59,-8) node[scale=5.0] {$u_3$};
\draw[color=black] (63.5,-8) node[scale=5.0] {$u_4$};
\draw[color=black] (49,2) node[scale=5.0] {$u_5$};
\draw[color=black] (63,2) node[scale=5.0] {$u_6$};
\draw[color=black] (58,10) node[scale=5.0] {$u_7$};

\draw[color=black] (76,-8) node[scale=5.0] {$v_1$};
\draw[color=black] (68,2) node[scale=5.0] {$v_2$};
\draw[color=black] (78,2) node[scale=5.0] {$v_3$};
\draw[color=black] (84,2) node[scale=5.0] {$v_4$};
\draw[color=black] (78,10) node[scale=5.0] {$v_5$};
\draw[color=black] (68,16) node[scale=5.0] {$v_6$};
\draw[color=black] (84,16) node[scale=5.0] {$v_7$};

\end{tikzpicture}
\caption{Two extremal graphs $G_4$ and $G_5$ for $ex(7,C_4)$.}
\label{rysG4G5}
\end{center}
\end{figure}
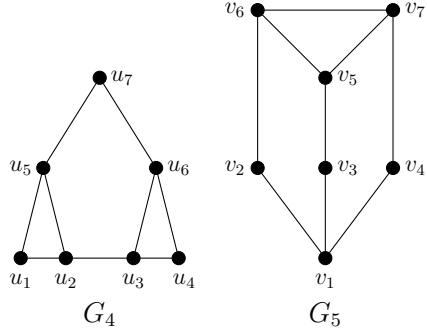

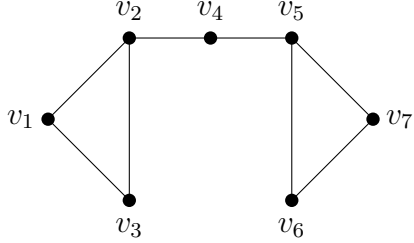
\begin{figure}[htb]
\begin{center}
\begin{tikzpicture}[scale=.18]

\fill [color=black] (4,6) circle (14pt);
\draw[color=black] (2,6) node {$v_1$};
\fill [color=black] (10,0) circle (14pt);
\draw[color=black] (10,-2) node {$v_3$};
\fill [color=black] (10,12) circle (14pt);
\draw[color=black] (10,14) node {$v_2$};
\fill [color=black] (16,12) circle (14pt);
\draw[color=black] (16,14) node {$v_4$};
\fill [color=black] (22,12) circle (14pt);
\draw[color=black] (22,14) node {$v_5$};
\fill [color=black] (22,0) circle (14pt);
\draw[color=black] (22,-2) node {$v_6$};
\fill [color=black] (28,6) circle (14pt);
\draw[color=black] (30,6) node {$v_7$};

\draw (4,6) -- (10,0)
(4,6) -- (10,12)
(10,0) -- (10,12)
(10,12) -- (16,12)
(16,12) -- (22,12)
(22,12) -- (22,0)
(22,12) -- (28,6)
(22,0) -- (28,6);
\end{tikzpicture}
\caption{The complement of the graph $F_7$.} 
\label{rysC7}
\end{center}
\end{figure}

Next, we will show that $r^{*}(P_3,C_7) \leq 13$. Let $F_7$ be the complement of the graph shown in Figure \ref{rysC7}. To prove that $F_7 \rightarrow (P_3,C_7)$, let $\chi$ be any $2$-coloring of edges of $F_7$ such that there is no red $P_3$ in $F_7$. We will show that the coloring~$\chi$ will imply a blue $C_7$ in $F_7$. To do so, consider vertex $v_4$. There are $4$ edges incidence to this vertex, at most one of them can be colored by red. Up to the symmetry of $F_7$, without loss of generality, we can assume that $v_1v_4$ is red or all edges $v_iv_4$, $i \in \{1,3,6,7\}$ are blue. Nonexistence a red $P_3$ forces the red edges to be a matching and that it suffices to consider maximum matchings. Then, using symmetries, there are only five subcases to discuss.

\begin{enumerate}
\item [1.] Edge $v_1v_4$ is red.

\begin{enumerate}
\item [1.1] if $v_2v_5$ and $v_3v_6$ is red, then $v_1, v_5, v_3, v_4, v_6, v_2, v_7, v_1$ is the blue cycle,
\item [1.2] if $v_2v_6$ and $v_3v_5$ are red, then $v_1, v_5, v_2, v_7, v_4, v_3, v_6, v_1$ is the blue cycle,
\item [1.3] if $v_2v_6$ and $v_3v_7$ are red, then the cycle $v_1, v_6, v_4, v_3, v_5, v_2, v_7, v_1$ is blue.
\end{enumerate}

\item [2.] All edges $v_iv_4$, $i \in \{1,3,6,7\}$ are blue. Then we have two subcases:

\begin{enumerate}
\item [2.1] if $v_2v_5$, $v_1v_6$, $v_3v_7$ are red, then we obtain the following blue cycle: $v_1, v_5, v_3, v_4, v_6, v_2, v_7, v_1$,

\item [2.2] if $v_2v_6$, $v_1v_5$, $v_3v_7$ are red, then the cycle: $v_1, v_4, v_6, v_3, v_5, v_2, v_7, v_1$ is blue.

\end{enumerate}

\end{enumerate}
\noindent For all cases, there is always a blue $C_7$, so $F_7 \rightarrow (P_3, C_7)$ and the proof is complete.
\end{proof}

\subsection{Upper bounds for $r^{*}(P_3,C_n)$}
In \cite{dbu1} Silaban \emph{et al.} proved that $r^{*}(P_3, C_n) \leq 2n-1$.
In this section we will show that this upper bound can be improved and we prove the following theorem.

\begin{thm}
For $n \geq 12$, $n$ is even,
$$ r^{*}(P_3, C_n) \leq 2n-2.$$
\label{twg}
\end{thm}

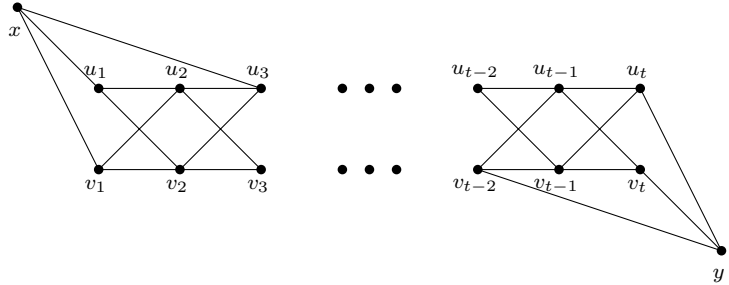
\begin{figure}[htb]
\begin{center}
\begin{tikzpicture}[scale=.18]

\begin{scriptsize}
\fill [color=black] (6,0) circle (10pt);
\draw[color=black] (5.8,-1.2) node {$v_1$};
\fill [color=black] (12,0) circle (10pt);
\draw[color=black] (11.8,-1.2) node {$v_2$};
\fill [color=black] (18,0) circle (10pt);
\draw[color=black] (17.8,-1.2) node {$v_3$};
\fill [color=black] (24,0) circle (10pt);
\fill [color=black] (26,0) circle (10pt);
\fill [color=black] (28,0) circle (10pt);
\fill [color=black] (34,0) circle (10pt);
\draw[color=black] (33.8,-1.2) node {$v_{t-2}$};
\fill [color=black] (40,0) circle (10pt);
\draw[color=black] (39.8,-1.2) node {$v_{t-1}$};
\fill [color=black] (46,0) circle (10pt);
\draw[color=black] (45.8,-1.2) node {$v_t$};
\fill [color=black] (6,6) circle (10pt);
\draw[color=black] (5.8,7.2) node {$u_1$};
\fill [color=black] (12,6) circle (10pt);
\draw[color=black] (11.8,7.2) node {$u_2$};
\fill [color=black] (18,6) circle (10pt);
\draw[color=black] (17.8,7.2) node {$u_3$};
\fill [color=black] (24,6) circle (10pt);
\fill [color=black] (26,6) circle (10pt);
\fill [color=black] (28,6) circle (10pt);
\fill [color=black] (34,6) circle (10pt);
\draw[color=black] (33.8,7.2) node {$u_{t-2}$};
\fill [color=black] (40,6) circle (10pt);
\draw[color=black] (39.8,7.2) node {$u_{t-1}$};
\fill [color=black] (46,6) circle (10pt);
\draw[color=black] (45.8,7.2) node {$u_t$};

\fill [color=black] (0,12) circle (10pt);
\draw[color=black] (-0.2,10.2) node {$x$};
\fill [color=black] (52,-6) circle (10pt);
\draw[color=black] (51.8,-7.8) node {$y$};

\draw
(6,0) -- (12,0)
(6,0) -- (12,6)
(6,0) -- (0,12)
(12,0) -- (6,6)
(12,0) -- (18,6)
(12,0) -- (18,0)
(18,0) -- (12,6)
(34,0) -- (40,0)
(34,0) -- (40,6)
(34,0) -- (52,-6)
(40,0) -- (34,6)
(40,0) -- (46,6)
(40,0) -- (46,0)
(46,0) -- (40,6)
(46,0) -- (52,-6)
(6,6) -- (12,6)

(6,6) -- (0,12)
(12,6) -- (18,6)
(18,6) -- (0,12)
(34,6) -- (40,6)
(40,6) -- (46,6)
(46,6) -- (52,-6)
;

\end{scriptsize}

\end{tikzpicture}
\caption{The graph $F_{n} \rightarrow (P_3,C_{n})$ for $n \geq 12$ and $n$ is even, $t = \frac{n-2}{2}$.}
\label{rysC12}
\end{center}
\end{figure}

\begin{proof}
Let $t = \frac{n-2}{2}$ and let $F_{n}$ be a graph with $$V(F_n) = \{x, y\} \cup \{u_i, v_i|i = 1, \ldots , t\}$$ and $$E(F_n) = \{xu_1, xv_1, xu_3, u_ty, v_ty, v_{t-2}y\}\cup S,$$ where
$$S = \{u_iu_{i+1}, v_iv_{i+1}, v_iu_{i+1}, u_iv_{i+1}|i = 1,\ldots , t - 1\}$$ (see Fig. \ref{rysC12}). In order to prove that $F_{n} \rightarrow (P_3,C_{n})$, let $\chi$ be any 2-coloring of edges of $F_{n}$ such that there is no red $P_3$ in $F_{n}$. We will show that the coloring $F_{n}$ will imply a blue $C_n$ in $F_{n}$.

\vspace{0.15cm}

\textbf{FACT 1.} Observe that if we have any two independent blue paths to $u_i$ and $v_i$, then we can extend these paths step by step to vertices $u_j$ and to $v_j$ for $1\leq i < j\leq t$. To do so, let us consider the vertex $u_i$. Since under the coloring $\chi$ there is no red $P_3$, at most one of edges $\{u_iu_{i+1}, u_iv_{i+1}\}$ can be red. If $u_iu_{i+1}$ is red, then $\{u_iv_{i+1}, v_iu_{i+1}\}$ must be blue. Using these 2 blue edges, we can extend our blue paths to $u_{i+1}$ and $v_{i+1}$, independently. If $u_iv_{i+1}$ is red, then $\{u_iu_{i+1}, v_iv_{i+1}\}$ must be blue. Using these 2 blue edges, we also can extend our blue paths to $u_{i+1}$ and $v_{i+1}$, independently. We can do the same process to extend our blue paths until reaching $u_j$ and  $v_j$.

\vspace{0.15cm}

\textbf{FACT 2.} There are always two independent blue paths from $x$ to $u_i$ and from $x$ to $v_i$ for $i = 1$ or $i = 3$. To prove this fact, let us consider the \mbox{the vertex $x$}. There are 3 incident edges to this vertex, at most one of them can be colored by red. Up to the symmetry of $F_{n}$, we
can assume that at most one edge of set $\{xu_1, xu_3\}$ is red.

If $xu_3$ is red, then $xu_1$ and $xv_1$ must be blue, therefore we have two blue paths from $x$ to $u_1$ and from $x$ to $v_1$. Note that a similar situation occurs if none of edges incidence to $x$ is red.

Now we can assume that $xu_1$ is red. In this case $xv_1$ and $xv_3$ are blue so we have one path from $x$ to $u_3$. We will construct a path of size $6$ with the set $\{u_1,u_2,v_1,v_2\}$) as inner vertices, namely the path from $x$ to $v_3$.
To do this consider the vertex $u_2$. Under the coloring $\chi$, at most one of edges $\{u_2u_3, u_2v_3, u_2v_1\}$ can be red. In all cases we obtain one among two possible blue paths from $x$ to $v_3$, namely $xv_1u_2u_1v_2v_3$ or $xv_1v_2u_1u_2v_3$.

Similarly, using the symmetry of $F_n$, we get two independent blue paths from $y$ to $u_j$ and from $y$ to $v_j$ for $j = t$ or $j = t-2$.

\vspace{0.5cm}

By using Fact 1 and 2, we obtain a blue cycle $C_n$ in $F$. Observe that the theorem holds for $3 \leq t-2$ and $n \geq 12$.
\end{proof}

Silaban \emph{et al.} \cite{dbu1} gave the upper bound for the restricted size Ramsey number of $P_3$ versus $P_n$. They proved that for even $n > 8$,
$r^{*}(P_3, P_n) \leq 2n - 1$. From the proof of Theorem \ref{twg} we see that if we delete edge $xu_3$ then for any 2-coloring of edges of $F_n \backslash \{xu_3\}$ that avoid red $P_3$, it must imply a blue $P_n$ in $F_n$. It means
we get a better upper bound of the restricted size Ramsey number for
$P_3$ versus $P_n$, $n \geq 12$ is even, as given in the following corollary.

\begin{cor}
For $n \geq 12$ and $n$ is even, $r^{*}(P_3, P_n) \leq 2n - 3$.
\end{cor}

\subsection{Computational Approach}

In this subsection we use a computational approach to determine the exact values of $r^{*}(P_3,C_n)$, $8 \leq n \leq 12$. We use the following Algorithm \ref{alg1} to find such numbers.

\begin{algorithm}
\caption{Deciding whether graph $F \rightarrow (P_3, C_n)$ or not}
\label{alg1}
\begin{algorithmic}[1]
\Statex \textbf{Require: } Adjacency matrix of biconnected graph $F$ on $n$ vertices.
\Statex \textbf{Ensure: } $F \rightarrow (P_3, C_n)$ or $F \nrightarrow (P_3, C_n)$.
\For {$m= \lfloor \frac{n}{2} \rfloor -2 \to  \lfloor \frac{n}{2} \rfloor$}
\For {every subset $S$ of $m$ edges that compose independent edge set}
\State $F' = F - S$
\State find a Hamiltonian cycle in $F'$
\If {no Hamiltonian cycle in $F'$} \Return {$F \nrightarrow (P_3, C_n)$, Break.} \EndIf
\EndFor
\EndFor
\State \textbf{return} $F \rightarrow (P_3, C_n)$
\end{algorithmic}
\end{algorithm}

We generate all the adjacency matrices of biconnected graphs with $n$ vertices ($8 \leq n \leq 12$) with minimum degree $3$ by using a program called {\it geng} \cite{MP}.

\begin{table}
\begin{center}
\begin{tabular}{|c||c||c||c||c||c|}
\hline
$n$ & $8$ & $9$ & $10$ & $11$ & $12$ \\
\hline
$r^{*}(P_3,C_n)$ & $15$ & $17$ & $18$ & $20$ & $22$ \\
\hline
$\#\{F \rightarrow (P_3, C_n) , |E (F) | = r^{*}(P_3,C_n)\}$
 & $10$ & $16$ & $2$ & $4$ & $8$ \\
\hline
\end{tabular}
\caption{Restricted size Ramsey numbers $r^{*}(P_3,C_n)$, $8 \leq n \leq  12$.}
\label{tab1}
\end{center}
\end{table}

From the above algorithm, we obtain the results which are presented in Table \ref{tab1}. This table provides the value of $r^{*}(P_3,C_n)$ and the number of non-isomorphic graphs $F$ of order $n$ and size $r^{*}(P_3,C_n)$ such that $F \rightarrow (P_3, C_n)$. Based on computer calculations, it turned out that the value of $m \in \{\lfloor \frac{n}{2} \rfloor -2, \lfloor \frac{n}{2} \rfloor -1,   \lfloor \frac{n}{2} \rfloor \}$. Examples of such graphs are presented in Fig. \ref{rysC9}, \ref{rysC10}, \ref{rysC11} and \ref{rysC12}. For the number $r^{*}(P_3,C_8)$ an example is a graph $K_{4,4}-e$.

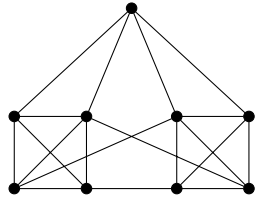
\begin{figure}[htb]
\begin{center}
\begin{tikzpicture}[scale=.12, transform shape]
\tikzstyle{every node}=[draw, circle, minimum size=33 pt,  fill=black];
\node (v1) at (0,0) {};
\node (v2) at (8,0) {} ;
\node (v3) at (0,8) {};
\node (v4) at (8,8) {};
\node (v5) at (18,0) {};
\node (v6) at (26,0) {};
\node (v7) at (18,8) {};
\node (v8) at (26,8) {};
\node (v9) at (13,20) {};
\draw (v1) -- (v2)
(v1) -- (v3)
(v1) -- (v4)
(v2) -- (v3)
(v2) -- (v4)
(v3) -- (v4)
(v5) -- (v6)
(v5) -- (v7)
(v5) -- (v8)
(v6) -- (v7)
(v6) -- (v8)
(v7) -- (v8)
(v4) -- (v6)
(v1) -- (v7)
(v3) -- (v9)
(v4) -- (v9)
(v7) -- (v9)
(v8) -- (v9)
(v2) -- (v5);
\end{tikzpicture}
\caption{Complement of the graph $F_9 \rightarrow (P_3,C_9).$}
\label{rysC9}
\end{center}
\end{figure}

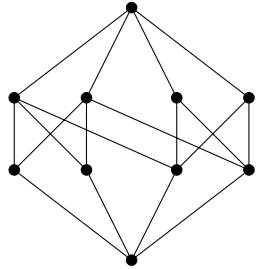
\begin{figure}[htb]
\begin{center}
\begin{tikzpicture}[scale=.12, transform shape]
\tikzstyle{every node}=[draw, circle, minimum size=33 pt, fill=black];
\node (v1) at (13,0) {};
\node (v2) at (0,10) {} ;
\node (v3) at (8,10) {};
\node (v4) at (18,10) {};
\node (v5) at (26,10) {};
\node (v6) at (0,18) {};
\node (v7) at (8,18) {};
\node (v8) at (18,18) {};
\node (v9) at (26,18) {};
\node (v10) at (13,28) {};
\draw (v1) -- (v2)
(v1) -- (v3)
(v1) -- (v4)
(v1) -- (v5)
(v2) -- (v6)
(v2) -- (v7)
(v3) -- (v6)
(v3) -- (v7)
(v4) -- (v8)
(v4) -- (v9)
(v5) -- (v8)
(v5) -- (v9)
(v4) -- (v6)
(v5) -- (v7)
(v6) -- (v10)
(v7) -- (v10)
(v8) -- (v10)
(v9) -- (v10);
\end{tikzpicture}
\caption{Graph $F_{10} \rightarrow (P_3,C_{10}).$}
\label{rysC10}
\end{center}
\end{figure}

\begin{figure}[htb]
\begin{center}
\begin{tikzpicture}[scale=.12, transform shape]
\tikzstyle{every node}=[draw, circle, minimum size=33 pt, fill=black];
\node (v1) at (10,0) {};
\node (v2) at (18,0) {} ;
\node (v3) at (26,0) {};
\node (v4) at (34,0) {};
\node (v5) at (10,8) {};
\node (v6) at (18,8) {};
\node (v7) at (26,8) {};
\node (v8) at (34,8) {};
\node (v9) at (2,14) {};
\node (v10) at (42,14) {};
\node (v11) at (22,20) {};

\draw (v1) -- (v2)
(v1) -- (v6)
(v1) -- (v9)
(v2) -- (v3)
(v2) -- (v5)
(v2) -- (v7)
(v3) -- (v4)
(v3) -- (v6)
(v3) -- (v8)
(v4) -- (v7)
(v4) -- (v10)
(v5) -- (v9)
(v5) -- (v6)
(v6) -- (v7)
(v6) -- (v11)
(v7) -- (v8)
(v7) -- (v11)
(v8) -- (v10)
(v9) -- (v11)
(v10) -- (v11);
\end{tikzpicture}
\caption{Graph $F_{11} \rightarrow (P_3,C_{11}).$}
\label{rysC11}
\end{center}
\end{figure}
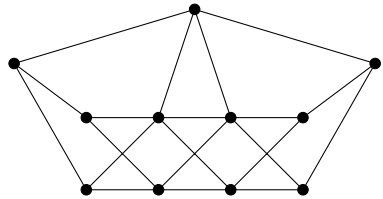

\section{Conclusion}
In this paper we established six new restricted size Ramsey numbers $r^{*}(P_3,C_n)$ for $7 \leq n \leq 12$. In addition, we gave the new upper bound for $n \geq 10$ and $n$ is even. It follows that the first open case of $r^{*}(P_3,C_n)$ is now $r^{*}(P_3,C_{13})$ and is certainly worth of further investigation. Based on results known earlier and described in this work as well as computer experiments for some bipartite graphs that are not presented here, let us formulate the following conjecture.

\begin{con}
For all $n \geq 10$, we have
$$r^{*}(P_{3},C_{n})=2n-2.$$
\end{con}

\section{Acknowledgment}
We would like to thank the student of the University of Gda\'nsk Maciej Godek for the independent performance of some computer experiments that confirmed the correctness of the results contained in the article.

\end{document}